\documentclass{amsart}
\usepackage{amsmath,amsthm,enumerate,graphicx}

\theoremstyle{theorem}
\newtheorem{theorem}{Theorem}
\newtheorem{corollary}[theorem]{Corollary}

\newtheorem{proposition}[theorem]{Proposition}

\theoremstyle{definition}

\title[Pressure of box-like IFS]{An explicit formula for the pressure of box-like affine iterated function systems}

\begin{document}
\author{Ian D. Morris}
\begin{abstract}
In this article we investigate the pressure function and affinity dimension for iterated function systems associated to the ``box-like'' self-affine fractals investigated by D.-J. Feng, Y. Wang and J.M. Fraser. Combining previous results of V. Yu. Protasov, A. K\"aenm\"aki and the author we obtain an explicit formula for the pressure function which makes it straightforward to compute the affinity dimension of box-like self-affine sets. We also prove a variant of this formula which allows the computation of a modified singular value pressure function defined by J.M. Fraser. We give some explicit examples where the Hausdorff and packing dimensions of a box-like self-affine fractal may be easily computed.
 \end{abstract}
 %MSC codes: 28A80 (primary), 15A60, 37D35, 37H15 (secondary)
\maketitle
\section{Introduction and statement of results}

 An \emph{iterated function system} or \emph{IFS} is defined to be any finite collection $T_1,\ldots,T_N$ of contractions of some fixed complete, nonempty metric space. It is a classical result of J. E. Hutchinson \cite{Hu81} that for every iterated function system $T_1,\ldots,T_N$ there exists a unique compact nonempty set $X$ which satisfies $X=\bigcup_{i=1}^N T_iX$. We call $X$ the \emph{attractor} of the IFS $T_1,\ldots,T_N$. In this article we shall be interested in the case where each $T_i$ is an affine contraction of $\mathbb{R}^d$, in which case the attractor $X$ is called a \emph{self-affine set}.
 
 The dimension theory of self-affine sets has been the subject of ongoing research investigation since the 1980s (see e.g. \cite{Be84,Fa88,Mc84}) and has enjoyed an intense burst of activity in recent years (we note for example \cite{Ba07,BaKaKo17,BaRa17,DaSi16,FaKe17,FeSh14,JoPoSi07,KaMo16,Mo16,Ra17}). In many cases the value of the Hausdorff, packing or box dimension of a self-affine set is determined by the values of a \emph{pressure functional}, which is itself defined in terms of limits of sequences of matrix products arising from the corresponding affine transformations $T_i$. The calculation of the dimension values defined by those formulas is in general a nontrivial problem: for example, it was shown only in 2014 that the \emph{affinity dimension}, a dimension formula first defined by K. Falconer in the 1988 article \cite{Fa88}, depends continuously on the affine transformations $T_i$ which are used to define it (see \cite{FeSh14}). It was also shown only in 2016 that the affinity dimension can in principle be computed to any prescribed accuracy in finitely many computational steps (see \cite{Mo16}). The purpose of this article is to show that in the special case of the ``box-like'' affine IFS studied by D.-J. Feng, Y. Wang and J.M. Fraser in \cite{FeWa05,Fr12,Fr16} the affinity dimension admits a simple description which renders it essentially trivial to compute. 

For each $d \geq 1$ let $M_d(\mathbb{R})$ denote the vector space of all $d \times d$ real matrices. We recall that the \emph{singular values} of a matrix $A \in M_d(\mathbb{R})$, denoted $\alpha_1(A),\ldots,\alpha_d(A)$, are defined to be the non-negative square roots of the eigenvalues of the positive semidefinite matrix $A^TA$, listed in decreasing order with repetition in the case of multiple eigenvalues. We note that $\alpha_1(A)=\|A\|$ and $\prod_{i=1}^d \alpha_i(A)=|\det A|$ for all $A \in M_d(\mathbb{R})$, and $\alpha_d(A)=\|A^{-1}\|^{-1}$ whenever $A \in M_d(\mathbb{R})$ is invertible. We will say that $A\in M_d(\mathbb{R})$ is a \emph{generalised permutation matrix} if it has exactly one nonzero entry in every row and in every column, and denote the group of all $d \times d$ generalised permutation matrices by $P_d(\mathbb{R})$. If $A \in P_d(\mathbb{R})$ then it is easy to see that the singular values of $A$ are simply the absolute values of the $d$ nonzero entries of $A$, listed in decreasing order. Following \cite{Fa88}, for each real number $s \geq 0$ we define the \emph{singular value function} $\varphi^s \colon M_d(\mathbb{R}) \to [0,+\infty)$ by
\[\varphi^s(A):=\left\{\begin{array}{cl}\alpha_1(A)\alpha_2(A)\cdots \alpha_{\lfloor s\rfloor}(A) \alpha_{\lceil s \rceil}(A)^{s-\lfloor s \rfloor}&\text{if }0\leq s\leq d,\\ |\det A|^{\frac{s}{d}}&\text{if }s \geq d.\end{array}\right.\]
For each fixed value of $s \geq 0$ the singular value function is upper semi-continuous on $M_d(\mathbb{R})$, is continuous when restricted to the set of invertible matrices, and satisfies $\varphi^s(AB)\leq \varphi^s(A)\varphi^s(B)$ for all $A,B \in M_d(\mathbb{R})$. If $\mathsf{A}=(A_1,\ldots,A_N) \in M_d(\mathbb{R})^N$ and $s>0$ are given, we define the \emph{pressure} of $(A_1,\ldots,A_N)$ to be the limit
\[P(\mathsf{A},s):= \lim_{n \to \infty} \frac{1}{n}\log \sum_{i_1,\ldots,i_n=1}^N \varphi^s\left(A_{i_n}\cdots A_{i_1}\right) \in [-\infty,+\infty).\]
The pressure $P(\mathsf{A},s)$ exists by subadditivity, and is a real number whenever $A_1,\ldots,A_N$ are all invertible by virtue of the elementary inequality $\varphi^s(B)\geq \alpha_d(B)^s$ which holds for every $B \in M_d(\mathbb{R})$. We refer the reader to the article \cite{Fa88} for proofs of these statements.

If $T_1,\ldots,T_N \colon \mathbb{R}^d \to \mathbb{R}^d$ are affine contractions then we have $T_ix=A_ix+v_i$ for all $x \in \mathbb{R}^d$ where $A_1,\ldots,A_N \in M_d(\mathbb{R})$ and $v_1,\ldots,v_N \in \mathbb{R}^d$. Since the transformations $T_i$ are contractions we have $\|A_i\|<1$ for every $i$. The function $s \mapsto P(\mathsf{A},s)$ associated to $\mathsf{A}:=(A_1,\ldots,A_N)$ is hence easily seen to be strictly decreasing, and it is in fact continuous. The unique value of $s\geq 0$ such that $P(\mathsf{A},s)=0$ is called the \emph{affinity dimension} of $(T_1,\ldots,T_N)$ and has been investigated for its role in the dimension theory of self-affine fractals since its introduction by Falconer in 1988; see \cite{FaMi07,FaSl09,FeSh14,Fr15,Mo16,PoVy15} as well as the original article \cite{Fa88}. 

Following \cite{FeWa05}, an affine IFS $T_1,\ldots,T_N$ acting on $\mathbb{R}^d$ will be called \emph{box-like} if we may write each $T_i$ in the form $T_ix=A_ix+v_i$ where $A_i \in P_d(\mathbb{R})$ for $i=1,\ldots,N$. Let $e_1,\ldots,e_d$ denote the standard basis of $\mathbb{R}^d$; we observe that $A \in M_d(\mathbb{R})$ belongs to $P_d(\mathbb{R})$ if and only if there exist $a_1,\ldots,a_d \in \mathbb{R} \setminus \{0\}$ and $\pi \colon \{1,\ldots,d\}\to \{1,\ldots,d\}$ such that $Ae_i = a_i e_{\pi(i)}$ for every $i=1,\ldots,d$. Let us also write $\rho(A)$ for the spectral radius of a matrix $A \in M_d(\mathbb{R})$.

The contribution of this article is the following formula for the pressure of a box-like iterated function system:
\begin{theorem}\label{th:main}
Let $\mathsf{A}=(A_1,\ldots,A_N) \in P_d(\mathbb{R})^N$, let $0<s<d$ and define $k:=\lfloor s \rfloor$. For $i=1,\ldots,N$ let $\pi_i \colon \{1,\ldots,d\} \to \{1,\ldots,d\}$ be the unique permutation and $a^{(i)}_1,\ldots,a^{(i)}_d$ the unique real numbers such that $A_i e_j= a^{(i)}_j e_{\pi_i(j)}$ for every $i=1,\ldots,N$ and $j=1,\ldots,d$.

Let $\mathfrak{S}$ denote the set of all pairs $(S,\ell )$ where $S \subset \{1,\ldots,d\}$ has $k$ elements and where $\ell \in \{1,\ldots,d\}\setminus S$, and observe that $\mathfrak{S}$ has exactly $(d-k){d \choose k}$ elements. Let us relabel the basis for $\mathbb{R}^{(d-k){d \choose k}}$ as $\{e_{S,\ell} \colon (S,\ell )\in\mathfrak{S}\}$. For each $i=1,\ldots,N$ let us define a matrix $\hat{A}_i \in P_{(d-k){d \choose k}}(\mathbb{R})$ by
\[\hat{A}_i e_{S,\ell}= \left(\prod_{j \in S} \left|a^{(i)}_j\right|\right)\left|a^{(i)}_\ell \right|^{s-k} e_{\pi_i(S),\pi_i(\ell)}.\]
Then
\[P(\mathsf{A},s)=\log\rho\left(\sum_{i=1}^N \hat{A}_i\right).\]
\end{theorem}
Theorem \ref{th:main} arises from the combination of two prior results: first, a formula for a certain pressure-like function for positive matrices given by V. Yu. Protasov in \cite{Pr10}; and second, a simplified expression for the function $\varphi^s(A)$ which is valid for all $A \in P_d(\mathbb{R})$ and which was given by A. K\"aenm\"aki and the author in \cite{KaMo16}.

Whilst Theorem \ref{th:main} can be used to quickly compute the affinity dimension of a box-like IFS of any dimension, in the two-dimensional case it admits a particularly straightforward articulation:
\begin{corollary}\label{co:revco}
Let $(T_1,\ldots,T_N)$ be an affine iterated function system acting on $\mathbb{R}^2$, and let  $v_1,\ldots,v_N \in \mathbb{R}^2$ and $A_1,\ldots,A_N \in M_2(\mathbb{R})$ such that $T_ix=A_ix+v_i$ for every $x \in \mathbb{R}^2$. Suppose that for some integer $k \in \{0,\ldots,N\}$ we have
\[A_i =\left\{\begin{array}{cl} \begin{pmatrix}a_i&0\\0&d_i\end{pmatrix} &\text{when }1 \leq i \leq k\\
\begin{pmatrix}0&b_i\\c_i&0\end{pmatrix} &\text{when }k+1 \leq i \leq N\end{array}\right.
\]
where each $a_i,b_i,c_i,d_i$ is nonzero. Then the affinity dimension of $(A_1,\ldots,A_N)$ is the unique real number $s>0$ such that one of the following holds: either the matrix
\[\begin{pmatrix}\sum_{i=1}^k |a_i|^s & \sum_{i=k+1}^N |b_i|^s \\ \sum_{i=k+1}^N |c_i|^s & \sum_{i=1}^k |d_i|^s\end{pmatrix} \]
has spectral radius $1$, and $0<s \leq 1$; or the matrix
\[\begin{pmatrix}\sum_{i=1}^k |a_i|.|d_i|^{s-1} & \sum_{i=k+1}^N |b_i|.|c_i|^{s-1} \\ \sum_{i=k+1}^N |b_i|^{s-1}|c_i| & \sum_{i=1}^k |a_i|^{s-1}|d_i|\end{pmatrix}\]
has spectral radius $1$ and $1 \leq s \leq 2$; or $\sum_{i=1}^N |\det A_i|^{s/2}=1$ and $s \geq 2$.
\end{corollary}
The proof of Theorem \ref{th:main} and Corollary \ref{co:revco} are given in the following section. In \S3 we show how these ideas can be adapted to a modified pressure functional considered by J.M. Fraser in \cite{Fr12}, and at the end of the paper we present some examples. 

\section{Proof of Theorem \ref{th:main} and Corollary \ref{co:revco}}

The following result is a special case of a theorem of V. Yu. Protasov \cite[Theorem 1]{Pr10}; some related results may be found in \cite{Pr97,Zh98}. Since the proof is both short and simple we include it.
\begin{proposition}\label{pr:one}
Let $A_1,\ldots,A_N\in M_d(\mathbb{R})$ be non-negative matrices. Then
\[\lim_{n \to \infty} \left(\sum_{i_1,\ldots,i_n=1}^N \left\|A_{i_n}\cdots A_{i_1}\right\|\right)^{\frac{1}{n}}=\rho\left(\sum_{i=1}^N A_i\right).\]
\end{proposition}
\begin{proof}
Let $|B|$ denote the sum of the absolute values of the entries of $B\in M_d(\mathbb{R})$ and note that if $B_1,B_2$ are non-negative matrices then $|B_1+B_2|=|B_1|+|B_2|$. Clearly $|\cdot|$ is a norm on $M_d(\mathbb{R})$ and in particular is equivalent to the Euclidean operator norm $\|\cdot\|$. Using Gelfand's formula and non-negativity we may calculate
\begin{align*}
\rho\left(\sum_{i=1}^N A_i\right) &=\lim_{n \to \infty} \left\|\left(\sum_{i=1}^N A_i\right)^n\right\|^{\frac{1}{n}}=\lim_{n \to \infty} \left|\left(\sum_{i=1}^N A_i\right)^n\right|^{\frac{1}{n}}\\
&=\lim_{n \to \infty} \left| \sum_{i_1,\ldots,i_n=1}^N A_{i_n}\cdots A_{i_1}\right|^{\frac{1}{n}}\\
&=\lim_{n \to \infty}  \left(\sum_{i_1,\ldots,i_n=1}^N \left|A_{i_n}\cdots A_{i_1}\right|\right)^{\frac{1}{n}}\\
&=\lim_{n \to \infty}  \left(\sum_{i_1,\ldots,i_n=1}^N \left\|A_{i_n}\cdots A_{i_1}\right\|\right)^{\frac{1}{n}}
\end{align*}
as required.
\end{proof}
The following result was introduced in \cite[\S5]{KaMo16}, where it was used to investigate the equilibrium states of the pressure functional for box-like affine iterated function systems. We again include the proof since it is barely longer than the statement.
\begin{proposition}\label{pr:two}
Let $d \geq 2$, let $0<s<d$ and define $k:=\lfloor s \rfloor$. Let $\mathfrak{S}$ denote the set of all pairs $(S,\ell )$ where $S \subset \{1,\ldots,d\}$ has $k$ elements and where $\ell \in \{1,\ldots,d\}\setminus S$, and observe that $\mathfrak{S}$ has exactly $(d-k){d \choose k}$ elements. Let us relabel the basis for $\mathbb{R}^{(d-k){d \choose k}}$ as $\{e_{S,\ell} \colon (S,\ell )\in\mathfrak{S}\}$.  

Define a function $\mathfrak{p}_s \colon P_d(\mathbb{R}) \to P_{(d-k){d \choose k}}(\mathbb{R})$ as follows. Given $A \in P_d(\mathbb{R})$, let $\pi \colon \{1,\ldots,d\} \to \{1,\ldots,d\}$ be the unique permutation and $a_1,\ldots,a_d$ the unique nonzero real numbers such that $Ae_i=a_ie_{\pi(i)}$ for every $i=1,\ldots,d$. Let $\mathfrak{p}_s(A) \in P_{(d-k){d \choose k}}(\mathbb{R})$ be the unique matrix such that
\[\mathfrak{p}_s(A) e_{S,\ell}= \left(\prod_{j \in S} \left|a_j\right|\right)\left|a_\ell \right|^{s-k} e_{\pi(S),\pi(\ell)}\]
for every $(S,\ell) \in \mathfrak{S}$. Then $\mathfrak{p}_s\colon P_d(\mathbb{R}) \to P_{(d-k){d \choose k}}(\mathbb{R})$ is a group homomorphism, and $\|\mathfrak{p}_s(A)\|=\varphi^s(A)$ for every $A \in P_d(\mathbb{R})$.
\end{proposition}
\begin{proof}
Let us first show that $\|\mathfrak{p}_s(A)\|=\varphi^s(A)$ for $A \in P_d(\mathbb{R})$. Let $Ae_i=a_ie_{\pi(i)}$ for each $i=1,\ldots,d$. We note that $\|\mathfrak{p}_s(A)\|$ is the largest of the absolute values of the entries of the generalised permutation matrix $\mathfrak{p}_s(A)$, and hence
\begin{align*}\|\mathfrak{p}_s(A)\| &= \max_{(S,\ell) \in \mathfrak{S}} \left(\prod_{j \in S} \left|a_j\right|\right)\left|a_\ell \right|^{s-k} \\
&= \max_{(S,\ell)\in \mathfrak{S}}  \left(\prod_{j\in S}\alpha_j(A)\right)\alpha_\ell(A)^{s-k}\\
& = \alpha_1(A)\cdots \alpha_k(A)\alpha_{k+1}(A)^{s-k}=\varphi^s(A)\end{align*}
as claimed. 

Now suppose that $A,B \in P_d(\mathbb{R})$ such that $Ae_i=a_ie_{\pi_A(i)}$ and $Be_i=b_ie_{\pi_B(i)}$ for each $i=1,\ldots,d$. Clearly $ABe_i=a_{\pi_B(i)}b_i e_{(\pi_A \circ\pi_B)(i)}$ for every $i=1,\ldots,d$. For every $(S,\ell) \in \mathfrak{S}$ we have
\[\mathfrak{p}_s(B)e_{S,\ell} =  \left(\prod_{j \in S} \left|b_j\right|\right)\left|b_\ell \right|^{s-k} e_{\pi_B(S),\pi_B(\ell)}\]
and therefore
\begin{align*}\mathfrak{p}_s(A)\mathfrak{p}_s(B)e_{S,\ell} &=\left(\prod_{j \in \pi_B(S)} \left|a_j\right|\right)\left(\prod_{j \in S} \left|b_j\right|\right)\left|a_{\pi_B(\ell)} \right|^{s-k}|b_\ell|^{s-k} e_{(\pi_A \circ \pi_B)(S),(\pi_A \circ \pi_B)(\ell)}\\
&=\left(\prod_{j \in S}\left|a_{\pi_B(j)} b_j\right|\right)\left|a_{\pi_B(\ell)} \right|^{s-k}|b_\ell|^{s-k} e_{(\pi_A \circ \pi_B)(S),(\pi_A \circ \pi_B)(\ell)}\\
& = \mathfrak{p}_s(AB)e_{S,\ell}.
\end{align*}
Since $(S,\ell) \in \mathfrak{S}$ was arbitrary we see that $\mathfrak{p}_s(AB)=\mathfrak{p}_s(A)\mathfrak{p}_s(B)$ as claimed.
\end{proof}
Theorem \ref{th:main} follows immediately from the two results above : given $A_1,\ldots,A_N\in P_d(\mathbb{R})$ and $s \in (0,d)$ we have
\begin{align*}e^{P(\mathsf{A},s)}&=\lim_{n \to \infty} \left(\sum_{i_1,\ldots,i_n=1}^N \varphi^s(A_{i_n}\cdots A_{i_1})\right)^{\frac{1}{n}}\\
& = \lim_{n \to \infty} \left(\sum_{i_1,\ldots,i_n=1}^N \left\|\mathfrak{p}_s(A_{i_n}\cdots A_{i_1})\right\|\right)^{\frac{1}{n}}\\
& = \lim_{n \to \infty} \left(\sum_{i_1,\ldots,i_n=1}^N \left\|\mathfrak{p}_s(A_{i_n})\cdots \mathfrak{p}_s(A_{i_1})\right\|\right)^{\frac{1}{n}} \end{align*}
by Proposition \ref{pr:two}, and since the matrices $\mathfrak{p}_s(A_i)$ have only non-negative entries,
\[ \lim_{n \to \infty} \left(\sum_{i_1,\ldots,i_n=1}^N \left\|\mathfrak{p}_s(A_{i_n})\cdots \mathfrak{p}_s(A_{i_1})\right\|\right)^{\frac{1}{n}} =\rho\left(\sum_{i=1}^N \mathfrak{p}_s(A_i)\right)\]
by Proposition \ref{pr:one}. The proof is complete.

Let us now derive Corollary \ref{co:revco}. Let $A_1,\ldots,A_N$ be as in the statement of that result. For each $s \geq 2$ we have
\begin{align*}P(\mathsf{A},s)&=\lim_{n \to \infty} \frac{1}{n}\log \sum_{i_1,\ldots,i_n=1}^N \varphi^s(A_{i_n}\cdots A_{i_1})\\
&=\lim_{n \to \infty} \frac{1}{n}\log \sum_{i_1,\ldots,i_n=1}^N \left| \det \left(A_{i_n}\cdots A_{i_1}\right)\right|^{\frac{s}{2}} =\log \sum_{i=1}^N |\det A_i|^{\frac{s}{2}}.\end{align*}
Since each $A_i$ is a contraction this quantity tends to $-\infty$ as $s \to \infty$, and since $P(\mathsf{A},s)$ is continuous, strictly decreasing as a function of $s$, and satisfies $P(\mathsf{A},0)=\log N \geq 0$ it follows that it has a unique zero, which is the affinity dimension. If the affinity dimension is given by $s \geq 2$ then clearly it solves $\sum_{i=1}^N |\det A_i|^{s/2}=1$ and we are in the third of the three cases mentioned in the statement of the corollary.

For every $s \in (0,1)$ the construction of Theorem \ref{th:main} leads to the basis $e_{\emptyset,1}$, $e_{\emptyset,2}$ for $\mathbb{R}^2$, with respect to which the matrices $\hat{A}_i$ take the form
\[\hat{A}_i =\left\{\begin{array}{cl} \begin{pmatrix}|a_i|^s&0\\0&|d_i|^s\end{pmatrix} &\text{when }1 \leq i \leq k\\
\begin{pmatrix}0&|b_i|^s\\|c_i|^s&0\end{pmatrix} &\text{when }k+1 \leq i \leq N\end{array}\right.
\]
and therefore we have
\begin{equation}\label{eq:eq1}e^{P(\mathsf{A},s)} = \rho\left(\begin{pmatrix}\sum_{i=1}^k |a_i|^s & \sum_{i=k+1}^N |b_i|^s \\ \sum_{i=k+1}^N |c_i|^s & \sum_{i=1}^k |d_i|^s\end{pmatrix}\right) \end{equation}
for all $s$ in this range. It follows that if the affinity dimension is in the range $0<s<1$ then it is the unique value for which the above spectral radius equals $1$; and if the affinity dimension is not in this range, then the above spectral radius must exceed $1$ for all $s \in (0,1)$.

For $s \in [1,2)$, the construction of Theorem \ref{th:main} leads to the basis $e_{\{1\},2}$, $e_{\{2\},1}$ for $\mathbb{R}^2$, with respect to which the matrices $\hat{A}_i$ take the form
\[\hat{A}_i =\left\{\begin{array}{cl} \begin{pmatrix}|a_i|.|d_i|^{s-1}&0\\0&|a_i|^{s-1}|d_i|\end{pmatrix} &\text{when }1 \leq i \leq k\\
\begin{pmatrix}0&|b_i|.|c_i|^{s-1}\\|b_i|^{s-1}|c_i|&0\end{pmatrix} &\text{when }k+1 \leq i \leq N\end{array}\right.
\]
We deduce that 
\begin{equation}\label{eq:eq2}e^{P(\mathsf{A},s)} =\rho\left(\begin{pmatrix}\sum_{i=1}^k |a_i|.|d_i|^{s-1} & \sum_{i=k+1}^N |b_i|.|c_i|^{s-1} \\ \sum_{i=k+1}^N |b_i|^{s-1}|c_i| & \sum_{i=1}^k |a_i|^{s-1}|d_i|\end{pmatrix}\right)\end{equation}
for all $s$ in this range. We observe that the expressions \eqref{eq:eq1} and \eqref{eq:eq2} coincide for $s=1$, and that \eqref{eq:eq2} evaluates to $\sum_{i=1}^k |a_id_i|+\sum_{i=k+1}^N |b_ic_i|=\sum_{i=1}^N |\det A_i|$ when $s=2$. It follows that if the affinity dimension is in $[1,2]$ then it is the unique value in that range for which the spectral radius \eqref{eq:eq2} equals $1$. This completes the proof of the corollary.

\section{A modified pressure function}

The methods of this article can easily be adapted to consider certain modifications of the pressure functional discussed in \S1. For example, let $(T_1,\ldots,T_N)$ be a box-like affine IFS acting on $\mathbb{R}^2$ and suppose that $(T_1,\ldots,T_N)$ satisfies the following \emph{Rectangular Open Set Condition}: there exists a nonempty open rectangle $R=(a,b)\times (c,d) \subset \mathbb{R}^2$ such that the sets $T_iR$ are disjoint open subsets of $R$. Suppose also that at least one of the transformations $T_i$ maps horizontal lines to vertical lines and vice versa. Then the attractor $X=\bigcup_{i=1}^N T_iX$ includes a (possibly reflected) image of itself rotated through $90^\circ$, which implies that the projection of $X$ onto each of the two co-ordinate axes has the same box dimension $t \in [0,1]$, say. By a theorem of J.M. Fraser \cite[Corollary 2.3]{Fr12}, in this case the packing dimension and box dimension of $X$ are equal to the unique real number $s \in [t,2t]$ such that
\[\lim_{n \to \infty} \left(\sum_{i_1,\ldots,i_n=1}^N \alpha_1(A_{i_n}\cdots A_{i_1})^t\alpha_2(A_{i_n}\ldots A_{i_1})^{s-t}\right)^{\frac{1}{n}}=1,\]
where of course each $A_i$ is the linear part of the associated affine transformation $T_i$. We note that if $t=1$ then $s$ is simply the affinity dimension. The above limit can also be computed using the methods of this article:
\begin{proposition}
Let $A_1,\ldots,A_N \in M_2(\mathbb{R})$, $0<t \leq 1$ and $t \leq s \leq 2t$. Suppose that for some integer $k \in \{0,\ldots,N\}$ we have
\[A_i =\left\{\begin{array}{cl} \begin{pmatrix}a_i&0\\0&d_i\end{pmatrix} &\text{when }1 \leq i \leq k\\
\begin{pmatrix}0&b_i\\c_i&0\end{pmatrix} &\text{when }k+1 \leq i \leq N\end{array}\right.
\]
where each $a_i,b_i,c_i,d_i$ is nonzero. Then the limit
\[\lim_{n \to \infty} \left(\sum_{i_1,\ldots,i_n=1}^N \alpha_1(A_{i_n}\cdots A_{i_1})^t\alpha_2(A_{i_n}\ldots A_{i_1})^{s-t}\right)^{\frac{1}{n}}\]
is equal to the spectral radius
\[\rho\left(\begin{pmatrix}\sum_{i=1}^k |a_i|^t|d_i|^{s-t} & \sum_{i=k+1}^N |b_i|^t|c_i|^{s-t} \\ \sum_{i=k+1}^N |b_i|^{s-t}|c_i|^t & \sum_{i=1}^k |a_i|^{s-t}|d_i|^t\end{pmatrix}\right).\]
\end{proposition}
\begin{proof}
Let us define $\mathfrak{q}_t \colon P_2(\mathbb{R}) \to P_2(\mathbb{R})$ by
\[\mathfrak{q}_t\left(\begin{pmatrix}a&0\\0&d\end{pmatrix}\right):=\begin{pmatrix}|a|^t|d|^{s-t}&0 \\ 0&|a|^{s-t}|d|^t\end{pmatrix}\]
and
\[\mathfrak{q}_t\left(\begin{pmatrix}0&b\\c&0\end{pmatrix}\right):=\begin{pmatrix}0&|b|^t|c|^{s-t}\\ |b|^{s-t}|c|^t&0\end{pmatrix}.\]
It is easily checked that $\mathfrak{q}_t(AB)=\mathfrak{q}_t(A)\mathfrak{q}_t(B)$ and $\|\mathfrak{q}_t(A)\|=\alpha_1(A)^t \alpha_2(A)^{s-t}$ for every $A,B \in P_2(\mathbb{R})$, and therefore
\begin{eqnarray*}
{\lefteqn{\lim_{n \to \infty} \left(\sum_{i_1,\ldots,i_n=1}^N \alpha_1(A_{i_n}\cdots A_{i_1})^t\alpha_2(A_{i_n}\ldots A_{i_1})^{s-t}\right)^{\frac{1}{n}}}}& &\\ 
& & =\lim_{n \to \infty} \left(\sum_{i_1,\ldots,i_n=1}^N \|\mathfrak{q}_t(A_{i_n})\cdots \mathfrak{q}_t(A_{i_1}) \|\right)^{\frac{1}{n}}=\rho\left(\sum_{i=1}^N \mathfrak{q}_t(A_i)\right)\end{eqnarray*}
similarly to the previous section.
\end{proof}
Our methods could also be extended to the computation of further variant pressure functionals considered in \cite{Fr16}, but we do not pursue this.

\section{Examples}
{\bf{Example 1.}} Define two affine transformations $T_1,T_2 \colon \mathbb{R}^2 \to \mathbb{R}^2$ by
\begin{align*}
T_1\begin{pmatrix}x\\y\end{pmatrix}&:=\begin{pmatrix}-\frac{13}{27}&0\\0&\frac{7}{9}\end{pmatrix}\begin{pmatrix}x\\y\end{pmatrix}+\begin{pmatrix}\frac{13}{27}\\\frac{2}{9}\end{pmatrix} \\
T_2\begin{pmatrix}x\\y\end{pmatrix}&:=\begin{pmatrix}0&\frac{13}{27}\\\frac{7}{9}&0\end{pmatrix}\begin{pmatrix}x\\y\end{pmatrix}+\begin{pmatrix}\frac{14}{27}\\0\end{pmatrix} 
\end{align*}
and denote the associated $2\times 2$ matrices by $A_1$ and $A_2$ respectively. For $0 <s \leq 1$ the matrix
\[\begin{pmatrix}\frac{13^s}{27^s}& \frac{13^s}{27^s}\\\frac{7^s}{9^2}& \frac{7^s}{9^s}\end{pmatrix}\]
\begin{figure}
    \centering
    \includegraphics[width=0.75\linewidth]{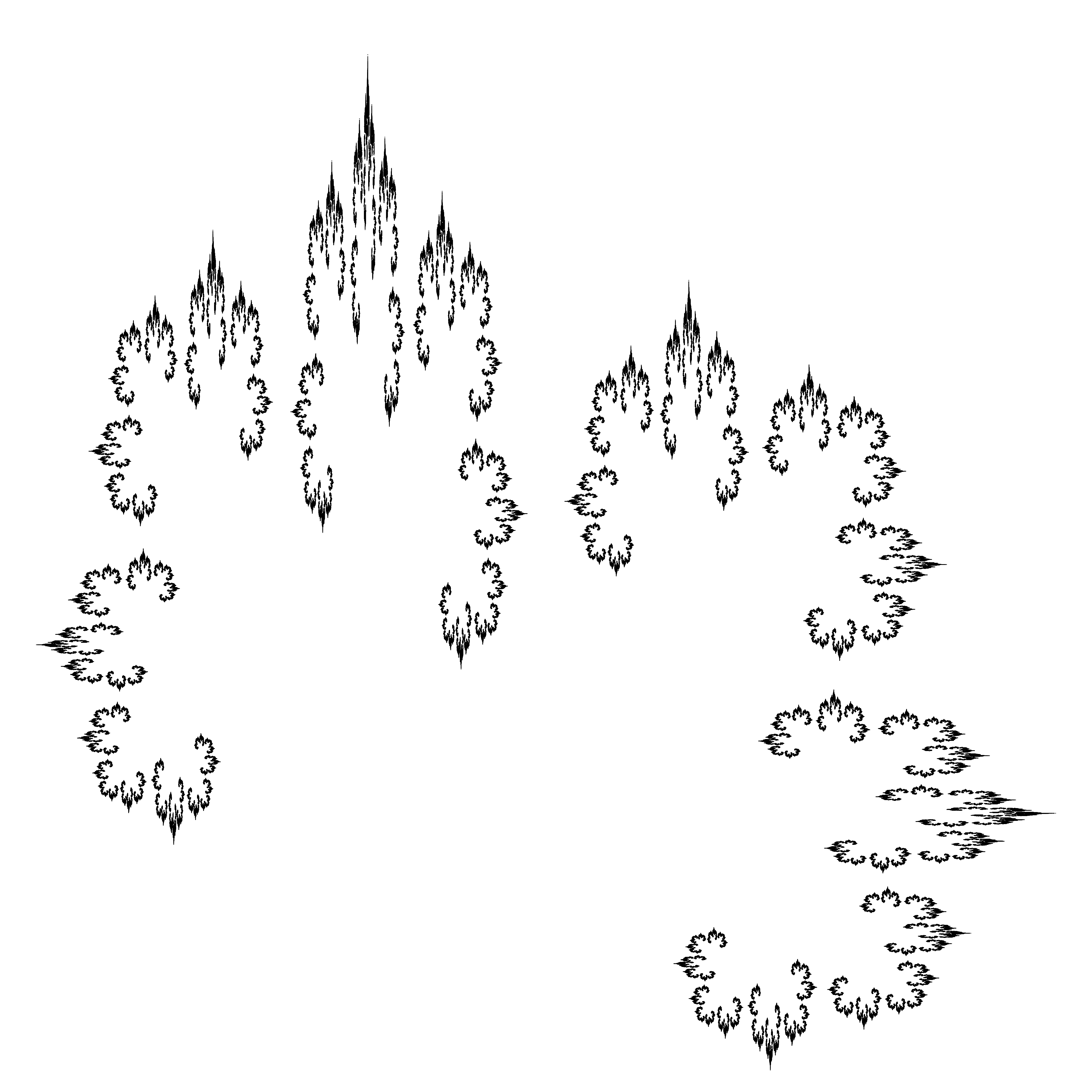}
    \caption{The attractor of the iterated function system defined in Example 1, for which the Hausdorff dimension, packing dimension and box dimension are all equal to the affinity dimension of the defining iterated function system.}
    \end{figure}
clearly has determinant zero, and therefore has spectral radius equal to its trace $\frac{13^s}{27^s}+ \frac{7^s}{9^s}\geq \frac{13}{27}+\frac{7}{9}>1$. We deduce that the affinity dimension of $(T_1,T_2)$ must exceed one. On the other hand clearly $|\det A_1|+|\det A_2|=\frac{218}{243}<1$ and so the affinity dimension must be less than $2$. It follows that the affinity dimension is the unique value of $s$ such that 
\[\begin{pmatrix}\frac{13}{27}\left(\frac{7}{9}\right)^{s-1}&\frac{13}{27}\left(\frac{7}{9}\right)^{s-1}\\ \frac{7}{9}\left(\frac{13}{27}\right)^{s-1}&\frac{7}{9}\left(\frac{13}{27}\right)^{s-1} \end{pmatrix}\]
has spectral radius $1$. Since this matrix also has determinant zero we conclude that the affinity dimension solves $\frac{13}{27}\left(\frac{7}{9}\right)^{s-1}+\frac{7}{9}\left(\frac{13}{27}\right)^{s-1}=1$ and is therefore equal to
\[s\approx 1.43035 20226 23969 408121 44729 61299 96697 74324 72301 14759 \ldots\]
Let us show using \cite[Proposition 7.2]{MoSh16} that the Hausdorff dimension of the attractor $X$ of $(T_1,T_2)$ has Hausdorff dimension equal to the affinity dimension of $(T_1,T_2)$. According to that proposition this holds true if we can show that:
\begin{enumerate}[(i)]
\item
The entries of the matrices and vectors defining $T_1$ and $T_2$ are algebraic;
\item
The IFS $(T_1,T_2)$ satisfies the \emph{Strong Open Set Condition}: there exists a nonempty open set $U \subset \mathbb{R}^2$ such that $T_1U, T_2U \subseteq U$ and $T_1U \cap T_2U=\emptyset$, and such that additionally $T_{i_n}\cdots T_{i_1}\overline{U}\subseteq U$ for some $i_1,\ldots,i_n \in \{1,2\}$;
\item
Let $P_x$, $P_y$ denote projection of $\mathbb{R}^2$ onto the first and second co-ordinates respectively; then for every $n\geq 1$, if $(i_1,\ldots,i_n), (j_1,\ldots,j_n) \in \{1,2\}^n$ are distinct then $P_xT_{i_n}\cdots T_{i_1}(0) \neq P_xT_{j_n}\cdots T_{j_1}(0)$ and $P_yT_{i_n}\cdots T_{i_1}(0) \neq P_yT_{j_n}\cdots T_{j_1}(0)$.
\end{enumerate}
Clearly (i) holds. To see that (ii) is satisfied we note that $T_1(0,1)^2=(0,\frac{13}{27})\times (\frac{2}{9},1)$ and $T_2(0,1)^2=(\frac{14}{27},1)\times (0,\frac{7}{9})$ are disjoint subsets of $(0,1)^2$, and that $T_2^2[0,1]^2=[\frac{14}{27},\frac{217}{243}]\times[\frac{98}{243},\frac{7}{9}]\subset(0,1)^2$. To prove (iii), suppose for a contradiction that $(i_1,\ldots,i_n),(j_1,\ldots,j_n) \in \{1,2\}^n$ are the shortest pair of distinct sequences such that either $P_xT_{i_n}\cdots T_{i_1}(0) = P_xT_{j_n}\cdots T_{j_1}(0)$ or $P_yT_{i_n}\cdots T_{i_1}(0) = P_yT_{j_n}\cdots T_{j_1}(0)$. By minimality of $n$ we necessarily have $i_n \neq j_n$, so without loss of generality suppose $i_n=1$ and $j_n=2$. Let $(x_1,y_1)^T:=T_{i_{n-1}}\cdots T_{i_1}(0)$ and $(x_2,y_2)^T:=T_{j_{n-1}}\cdots T_{j_1}(0)$. These vectors belong to $[0,1]^2$ and their entries are obtained from $0$ by repeatedly adding one of $\frac{13}{27},\frac{14}{27}$ or $\frac{2}{9}$ and by repeatedly multiplying by $\pm\frac{13}{27}$ or $\frac{7}{9}$, in some order. In particular each entry is of the form $p/q$ where $p \in \mathbb{Z}$ and $q$ is a power of three; which is to say, each entry belongs to the ring $\mathbb{Z}[\frac{1}{3}]$. By hypothesis we have either $P_x T_1(x_1,y_1)^T=P_x T_2(x_2,y_2)^T$ or $P_y T_1(x_1,y_1)^T=P_y T_2(x_2,y_2)^T$. If the former, we obtain $-\frac{13}{27}x_1 + \frac{13}{27} = \frac{13}{27}y_2+\frac{14}{27}$ which yields $x_1+y_2=-\frac{1}{13}<0$, an impossibility. If the latter we obtain $\frac{7}{9}y_1+\frac{2}{9}=\frac{7}{9}x_2$, whence $\frac{2}{7}=x_2-y_1 \in \mathbb{Z}[\frac{1}{3}]$ which contradicts the fundamental theorem of arithmetic. We conclude that criteria (i)--(iii) above are satisfied by $(T_1,T_2)$ and therefore the Hausdorff dimension (and hence also the packing and box dimensions) of $X$ are equal to the affinity dimension of $(T_1,T_2)$ as claimed.

{\bf{Example 2.}} Define four affine transformations $T_1,T_2,T_3,T_4 \colon \mathbb{R}^2 \to \mathbb{R}^2$ by
\begin{align*}
T_1\begin{pmatrix}x\\y\end{pmatrix}&:=\begin{pmatrix}\frac{1}{3}&0\\0&\frac{2}{3}\end{pmatrix}\begin{pmatrix}x\\y\end{pmatrix}+\begin{pmatrix}\frac{2}{3}\\0\end{pmatrix} \\
T_2\begin{pmatrix}x\\y\end{pmatrix}&:=\begin{pmatrix}-\frac{2}{3}&0\\0&-\frac{1}{3}\end{pmatrix}\begin{pmatrix}x\\y\end{pmatrix}+\begin{pmatrix}\frac{2}{3}\\1\end{pmatrix}\\
T_3\begin{pmatrix}x\\y\end{pmatrix}&:=\begin{pmatrix}0&\frac{2}{9}\\-\frac{1}{3}&0\end{pmatrix}\begin{pmatrix}x\\y\end{pmatrix}+\begin{pmatrix}\frac{2}{3}\\1\end{pmatrix} \\
T_4\begin{pmatrix}x\\y\end{pmatrix}&:=\begin{pmatrix}0&\frac{4}{9}\\-\frac{1}{3}&0\end{pmatrix}\begin{pmatrix}x\\y\end{pmatrix}+\begin{pmatrix}\frac{2}{9}\\\frac{2}{3}\end{pmatrix} 
\end{align*}\begin{figure}
    \centering
    \includegraphics[width=0.75\linewidth]{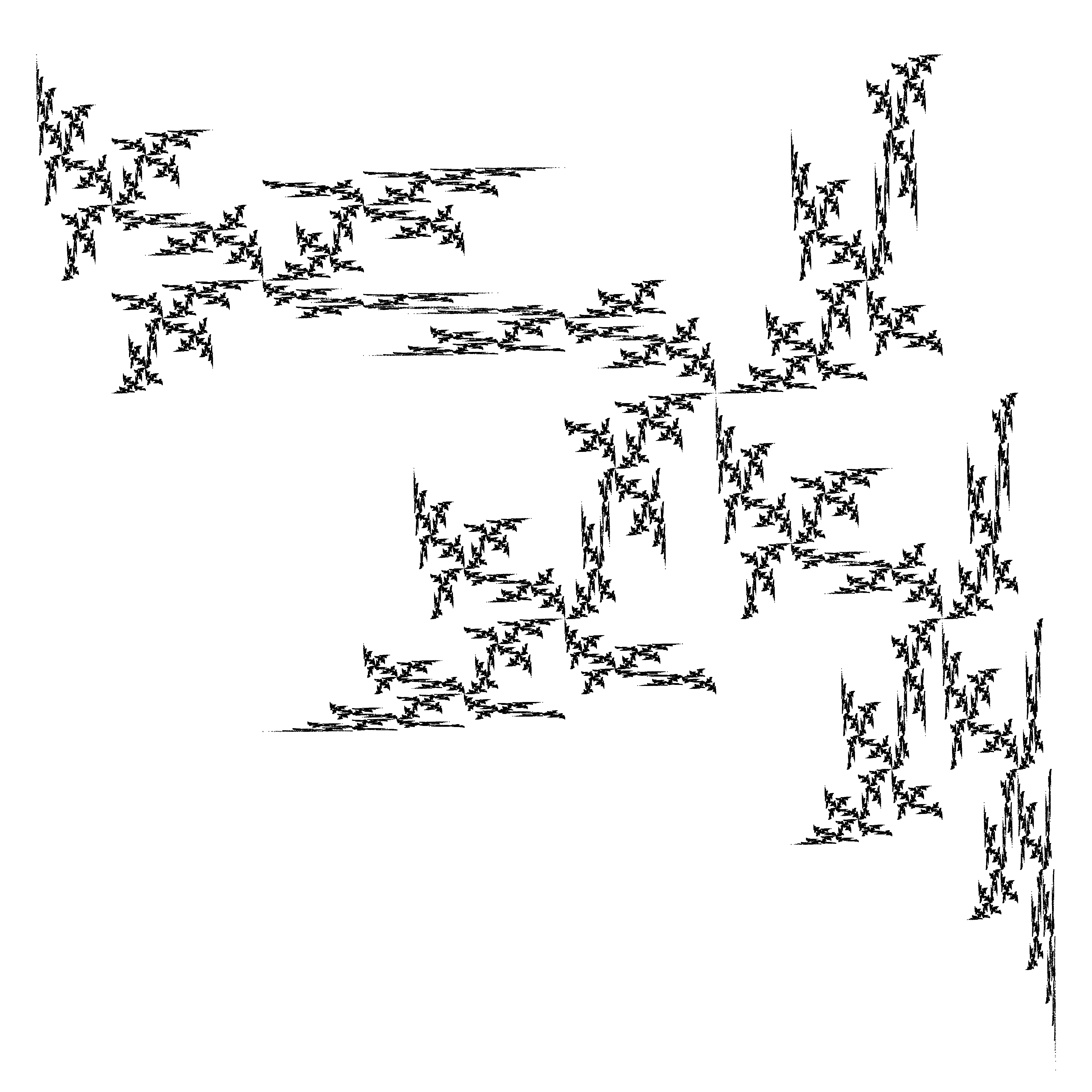}
    \caption{The attractor of the iterated function system defined in Example 2, which has packing dimension and box dimension equal to the affinity dimension of its defining iterated function system.}
    \end{figure} and denote the associated $2\times 2$ matrices by $A_1$, $A_2$, $A_3$ and $A_4$ respectively. It is easy to see that $(T_1,T_2)$ has affinity dimension $1$ since the matrix produced by applying Corollary \ref{co:revco} to $(A_1,A_2)$ alone is equal to the identity matrix when $s=1$. It follows that the affinity dimension of $(T_1,T_2,T_3,T_4)$ must exceed $1$. Since on the other hand $\sum_{i=1}^4 |\det A_i|=\frac{2}{3}<1$ the affinity dimension of $(T_1,T_2,T_3,T_4)$ must be less than $2$, and hence it is the unique value of $s \in [1,2]$ such that the matrix
\[\frac{1}{3^s}\begin{pmatrix} 2^{s-1} + 2 & 2 \\\left(\frac{2}{3}\right)^{s-1} +\left(\frac{4}{3}\right)^{s-1}&  2^{s-1}+2\end{pmatrix}\]
has spectral radius $1$. A real $2 \times 2$ matrix $B$ has $1$ as an eigenvalue if and only if $1+\det B=\mathrm{tr}\,B$ by elementary consideration of the characteristic polynomial, and hence the affinity dimension $s \in[1,2]$ solves
\[ 1 + \frac{1}{9^s}\left(4^{s-1}+2^{s+1}+4- 2\left(\frac{2}{3}\right)^{s-1}-2\left(\frac{4}{3}\right)^{s-1} \right)=\frac{2^s+4}{3^s}.\]
We may thus easily compute the affinity dimension of $(T_1,T_2,T_3,T_4)$ to be
\[s \approx 1.54202 66478 62956 03651 88932 87043 45802 50254 31511 44645\ldots\]%http://www.wolframalpha.com/input/?i=solve+(2%5Ex%2B4)%2F3%5Ex+%3D+1%2B(9%5E(-x))*(4%5E(x-1)%2B2%5E(x%2B1)%2B4-2*(2%2F3)%5E(x-1)-2*(4%2F3)%5E(x-1))
For this example we ignore the Hausdorff dimension in favour of the simpler task of studying the packing dimension of the attractor. It is easily verified that $(T_1,T_2,T_3,T_4)$ maps the open unit square $(0,1)^2$ into four pairwise disjoint open subrectangles of $(0,1)^2$ which are separated by the lines $x=\frac{2}{3}$ and $y=\frac{2}{3}$ and which meet at corners at the point $(\frac{2}{3},\frac{2}{3})$. In particular $(T_1,T_2,T_3,T_4)$ satisfies the Rectangular Open Set Condition of Feng-Wang and Fraser \cite{FeWa05,Fr12}. By \cite[Lemma 2.8]{Fr12} the projection of the attractor $X$ onto either co-ordinate axis has box dimension 1, and it follows by \cite[Corollary 2.6]{Fr12} that the packing dimension and box dimension of $X$ are equal to the affinity dimension of $(T_1,T_2,T_3,T_4)$.

\section{Acknowledgements}
This research was supported by the Leverhulme Trust (Research Project Grant number RPG-2016-194). The author thanks J.M. Fraser, A. K\"aenm\"aki and P. Shmerkin for helpful remarks.
\bibliographystyle{acm}
\bibliography{ox-like}
\end{document}